\RequirePackage[l2tabu, orthodox]{nag}

\documentclass[a4paper,reqno]{amsart}

\usepackage{lmodern}
\usepackage[T1]{fontenc}
\usepackage[utf8]{inputenc}
\usepackage[english]{babel}
\usepackage{microtype} 

\usepackage{amsmath,amssymb,amsthm,mathrsfs,latexsym,mathtools,mathdots,enumerate,tikz,ytableau}
\usetikzlibrary{positioning}
\usepackage[enableskew,vcentermath]{youngtab}
\usepackage{hyperref}
\usepackage[capitalize,noabbrev]{cleveref}
\usepackage{todonotes}

\usepackage{navigator}
\embeddedfile{Structure constants}{jackStructureConstants6.txt}

\newtheorem{theorem}{Theorem}
\newtheorem{proposition}[theorem]{Proposition}
\newtheorem{lemma}[theorem]{Lemma}

\newtheorem{remark}[theorem]{Remark}
\newtheorem{question}[theorem]{Question}
\newtheorem{conjecture}[theorem]{Conjecture}

\newcommand{\defin}[1]{\emph{#1}}

\newcommand{\xvec}{\mathbf{x}}

\newcommand{\arm}{\mathit{a}}
\newcommand{\leg}{\mathit{l}}

\newcommand{\jackJ}{J}
\newcommand{\jackP}{P}

\renewcommand{\a}{\alpha}
\newcommand{\la}{\lambda}

\newcommand{\Ch}{\vartheta^{(\a)}}
\DeclareMathOperator{\sh}{sh}
\DeclareMathOperator{\SYT}{SYT}
\DeclareMathOperator{\RSSYT}{RSSYT}

\title[Structure constants of shifted Jack functions]{A positivity conjecture on
the structure constants of shifted Jack functions}
\author[P.~Alexandersson]{Per Alexandersson}
\author[V.~Féray]{Valentin Féray}

\begin{document}

\begin{abstract}
We consider Jack polynomials $J_\la$ and their shifted analogue $J^\#_\la$.
In 1989, 
Stanley conjectured that $\langle J_\mu J_\nu, J_\lambda \rangle$ is a polynomial with nonnegative coefficients in the parameter $\alpha$.
In this note, we extend this conjecture to the case of \emph{shifted} Jack polynomials.
%
\end{abstract}

\maketitle

\section{Introduction}

In his seminal article on the combinatorics of Jack polynomials \cite{Stanley1989},
Stanley made a positivity conjecture on the \emph{Littlewood--Richardson type} coefficients
for Jack polynomials, i.e. the coefficients appearing in the Jack basis expansion of the product
of two Jack polynomials.
Despite intensive works on Jack polynomials on the one side and
on generalizations of the Littlewood--Richardson rule on the other side,
this problem still remains open.
\smallskip

In this note, we suggest an extension of Stanley's conjecture 
to \emph{shifted} Jack polynomials, \cref{conj:laurentjackLRcoeffs} below.
An advantage of the shifted version is that there exist a recursion formula
to compute the corresponding {Littlewood--Richardson type} coefficients (\cref{prop:structureConstantRecursion}),
giving another angle of attack for the conjecture.
We give numerical data and very partial results to support the conjecture:
a general polynomiality result in \cref{sec:poly}
and the analysis of some specific cases in \cref{sec:part}.
\medskip

Throughout this note, we use the notation of Macdonald's book \cite{Macdonald1995} 
regarding partitions, symmetric functions, (skew) diagrams, etc.
We assume the reader to be acquainted with basic notions in this field.

\section{Jack polynomials}

A \emph{reverse (semi-standard) tableau} of shape $\mu$ is a filling of the boxes of $\mu$
with numbers in $\{1,\dots,n\}$ which is weakly decreasing along rows and strictly decreasing along columns.
The set of such tableaux is denoted $\RSSYT(\mu,n)$.
For $T$ in $\RSSYT(\mu,n)$, we denote $\rho^i(T)$ the skew diagram consisting in boxes of $T$
with entry $i$. This is always a horizontal strip (no two boxes in the same column).

For a horizontal strip $\lambda/\mu$, let 
$R_{\lambda/\mu}$ (resp. $C_{\lambda/\mu}$) 
denote the set of boxes in a row (resp. column) that intersects the shape $\lambda/\mu$.
We then define
\begin{equation}\label{eq:psidef2}
\psi_{\lambda/\mu}(\alpha) = \prod_{s \in R_{\lambda/\mu} - C_{\lambda/\mu} } 
\frac{(\alpha \arm_\lambda(s) + \leg_\lambda(s) + \a)(\alpha \arm_\mu(s) + \leg_\mu(s) + 1)}{(\alpha \arm_\lambda(s) + \leg_\lambda(s) + 1)(\alpha \arm_\mu(s) + \leg_\mu(s) + \a)}.
\end{equation}

The Jack polynomial $\jackP_\mu$ in $n$ variables $x_1,\dots,x_n$ may be defined as\footnote{This formula is usually
found in the literature with a sum over semi-standard tableaux (and not reverse ones as here).
But since the resulting sum is known to be symmetric, this modification is irrelevant.
In the next section, using reverse tableaux is however important.}
\[
\jackP_\mu = \sum_{T \in \RSSYT(\mu,n)} \psi_T(\alpha) \prod_{s \in \mu} x_{T(s)}.
\] 

For $\alpha=1$, we have $\psi_{\lambda/\mu}(1)=1$ and 
$\jackP_\mu$ specializes to the Schur polynomial in $n$ variables when $\alpha=1$.

Let us define the two $\alpha$-deformations of the hook products, $H_\lambda$ and $H'_\lambda$, as
\[
H_\lambda = \prod_{s \in \lambda} (\alpha\arm_\lambda(s) + \leg_\lambda(s) + 1), \quad
H'_\lambda = \prod_{s \in \lambda} (\alpha \arm_\lambda(s) + \leg_\lambda(s) + \alpha).
\]
(these are called $c_\lambda(\alpha)$ and $c'_\lambda(\alpha)$ in \cite[(10.22)]{Macdonald1995}).
The integral version of the Jack polynomial $J_\mu$ is then given by
 $\jackJ_\mu = H_\mu \jackP_\mu$.

\section{Shifted Jack polynomials}

We define the \defin{shifted Jack polynomials} $\jackP^{\#}_\mu$ as 
\begin{equation}
\jackP^{\#}_\mu(x_1,\dots,x_n) 
= \sum_{T \in \RSSYT(\mu,n)} \psi_T(\alpha) \prod_{s \in \mu} (x_{T(s)} -  \arm'(s) + \leg'(s)/\alpha),
\label{EqOkounkovCombJsh}
\end{equation}
where $\arm'(s) = j-1$ and $\leg'(s) = i-1$ for the box $s=(i,j)$. Thus, the top degree component of 
$\jackP^{\#}_\mu$ is given by the ordinary Jack polynomial $\jackP_\mu$.
This definition can be found in \cite{OkounkovOlshanskiShiftedJack},
see also \cite{SahiShiftedJack,KnopSahiShiftedSym} for older references
on these polynomials (with a shift of variables).

\medskip 
Alternatively to the tableaux definition, 
the Jack shifted polynomials can also be characterized by their values $\jackP^{\#}_\mu(\lambda)$ on partitions,
i.e. when we specialize $x_j=\la_j$ (for $j \le n$; appending zeroes to $\lambda$ if necessary).
Let $\Lambda^{\alpha}(x_1,\dots,x_n)$ be the ring of polynomials which are \defin{shifted symmetric},
that is, symmetric in the variables $x_i - i/\alpha$.
The polynomials $\jackP^{\#}_\mu$ lie in $\Lambda^{\alpha}(x_1,\dots,x_n)$ and it 
was shown by F.~Knop and S.~Sahi in \cite{KnopSahiShiftedSym} that $\jackP^{\#}_\mu$ is 
the unique polynomial in $\Lambda^{\alpha}(x_1,\dots,x_n)$
such that
\begin{equation}\label{eq:vanishing}
\jackP^{\#}_\mu(\lambda) = 
\begin{cases}
\alpha^{-|\mu|} H'_\mu,   &\lambda = \mu \\
0, \quad &|\lambda| \leq |\mu|, \; \lambda \neq \mu.
\end{cases}
\end{equation}
Additionally, we have $\jackP^{\#}_\mu(\lambda) =0$ whenever $\lambda \not\supseteq \mu$.

There is a $J$-normalized version of shifted Jack polynomials given as  $\jackJ^{\#}_\mu = H_\mu  \jackP^{\#}_\mu$,
which is the main topic of study in \cite{LassalleConjecturePQ}, where the notation 
$\jackJ^{\dagger}_\mu$ is used in place of $\jackJ^{\#}_\mu$.

From the definition of $\jackP^{\#}_\mu$, 
it is easily seen that $\jackP^{\#}_\mu(x_1,\dots,x_n,0) =\jackP^{\#}_\mu(x_1,\dots,x_n)$
so we can see $\jackP^{\#}_\mu$ as shifted symmetric functions in infinitely many variables
and they define a basis in the projective limit $\Lambda^{\alpha}$
of the rings $\Lambda^{\alpha}(x_1,\dots,x_n)$.
In fact, for any fixed $d$, the family $(\jackP^{\#}_\mu)_{|\mu| \le d}$
is a basis of the space of shifted symmetric functions
of degree at most $d$.

\section{Shifted Jack structure constants}

We define the {\em Littlewood--Richardson type structure constants for shifed Jack functions},
$c^{\lambda}_{\mu\nu}$, which are rational functions in $\alpha$,
by the relation
\begin{equation}
  \label{eq:defLR}
\jackP_\mu^\# \jackP_\nu^\# = \sum_\lambda c^{\lambda}_{\mu\nu} \jackP_\lambda^\#, \text{ or equivalently }
\jackJ_\mu^\# \jackJ_\nu^\# = \sum_\lambda c^{\lambda}_{\mu\nu} \frac{H_\mu \cdot H_\nu}{H_\lambda}  \jackJ_\lambda^\#.
\end{equation}
Since $ \jackP_\mu^\# \jackP_\nu^\#$ is a shifted symmetric function of degree at most $\la$,
the above sum runs over partitions $\lambda$ of size at most $|\mu|+|\nu|$
(since the functions are not homogeneous, we need to take $|\lambda| \le |\mu|+|\nu|$ and we cannot restrict to
$|\lambda| = |\mu|+|\nu|$).
\medskip

For $\a=1$, these are Littlewood--Richardson coefficients for shifted Schur functions.
A combinatorial interpretation, implying their nonegativity has been 
found (in the more general context of factorial Schur functions)
by Molev \cite{Molev2009}; see also \cite{MolevSaganLRFactorialSchur}.
\medskip

On the other hand,
taking the homogeneous  component of degree $|\mu|+|\nu|$ in \eqref{eq:defLR}, we get
\[
\jackJ_\mu \jackJ_\nu = \sum_{\lambda: |\lambda|=|\mu|+|\nu|} 
    c^{\lambda}_{\mu\nu} \frac{H_\mu \cdot H_\nu}{H_\lambda} \jackJ_\lambda.
\]
Recall that Jack polynomials are orthogonal for the standard $\a$-deformation of Hall scalar product,
namely we have $ \langle \jackJ_\rho,\jackJ_\lambda \rangle = H_\la H'_\la \delta_{\rho,\lambda}$
(see, e.g. \cite[Thm. 5.8]{Stanley1989}).
This implies that for $|\lambda|=|\mu|+|\nu|$,
\[ \langle \jackJ_\mu \jackJ_\nu, \jackJ_\lambda \rangle = H_\mu H_\nu H'_\lambda c^{\lambda}_{\mu\nu}. \]
In his 1989 paper, Stanley conjectures that the above quantity $\langle \jackJ_\mu \jackJ_\nu, \jackJ_\lambda \rangle$
is a polynomial with nonnegative coefficients in $\a$, for all $|\lambda|=|\mu|+|\nu|$.
Unlike the left-hand side, the right-hand side is also defined for $|\lambda|<|\mu|+|\nu|$.
We conjecture that, up to multiplication by an appropriate power of $\alpha$,
the right-hand side is always a polynomial with nonnegative coefficients in $\a$.
This actually strengthens a conjecture of Sahi, which asserts
that $c^{\lambda}_{\mu\nu}$ is a quotient of two polynomials in $\a$ 
with nonnegative coefficients \cite{SahiLRCoefShiftedJack}.

\begin{conjecture}\label{conj:laurentjackLRcoeffs}
  For any partition $\lambda,\mu,\nu$ with $|\lambda| \le |\mu|+|\nu|$, the quantity
\begin{equation}\label{eq:structureconstantconjecture}
  g^{\lambda}_{\mu\nu} \coloneqq H_\mu H_\nu H'_\lambda c^{\lambda}_{\mu\nu}
\end{equation}
is a Laurent polynomial with nonnegative integer coefficients in $\alpha$.
\end{conjecture}
Here are a few examples of these coefficients
 \begin{align*}
   g_{31,21}^{421} &= 8\a^5(9 + 97\a + 294\a^2 + 321\a^3 + 131\a^4 + 12\a^5), \\
   g_{41,41}^{541} &= 144\a^8(1 + \a)^2(2 + \a)(2 + 3\a)(1 + 4\a)(18 + 149\a + 238\a^2 + 120\a^3),\\
   g_{2111,1111}^{2111} &=  \frac{144 (\a+1)^2 (\a+2)^2 (\a+4)(2 \a+3)^2}{\a}.  
 \end{align*}
In the first example, $|\lambda| = |\mu|+|\nu|$, so that this coefficient also writes
$\langle \jackJ_{31} \jackJ_{21}, \jackJ_{421} \rangle$. This value was given
as an example by Stanley (who credits Hanlon for the computation \cite[Section 8]{Stanley1989}).
The other examples however do not rewrite as such scalar products.
The third line shows that $g^{\lambda}_{\mu\nu}$ is not always a polynomial in $\a$.
We attach\footnote{To access the file with acrobat reader, click on view/navigation panels/attachments.
Alternatively, you can download the source files from \texttt{arXiv.org}.}
the values of all coefficients for $|\mu|,|\nu| \le 6$ supporting the above conjecture.

\section{A polynomiality result}
\label{sec:poly}
\begin{proposition}
The expression $\a^{|\mu|+|\nu|-|\la|-2} g^{\lambda}_{\mu\nu}$ is a polynomial 
with rational coefficients in $\alpha$.
\end{proposition}
It is possible that multiplying by some smaller and/or more natural power of $\a$
suffices to get a polynomial, but we have no other results in this direction.
\begin{proof}
Let us expand the Jack symmetric functions in the power-sum basis:
\begin{equation}
 J_\la = \sum_\mu \theta^{(\a)}_\mu(\la) \, p_\mu.
 \label{eq:Jp}
 \end{equation}
We then introduce the following quantity
\begin{equation}
\Ch_\mu(\la)= 
\begin{cases}
    |\lambda|(|\lambda|-1)\cdots (|\lambda|-|\mu|+1) \frac{\theta^{(\a)}_{\mu 1^{|\lambda|-|\mu|}}(\lambda)}
    {\theta^{(\a)}_{1^{|\mu|}}(\lambda)} &\text{ if }|\lambda| \ge |\mu|\\
    0 &\text{ if }|\lambda| < |\mu|.
\end{cases}
\label{EqDefCh1}
\end{equation}
It was proved in \cite{LassalleConjecturePQ} that there exists 
a linear isomorphism $\sh(\cdot)$ from usual symmetric functions to shifted symmetric ones
so that
\[ \sh(J_\la) = J^\#_\la, \quad \sh(p_\mu)= \a^{-(|\mu|-\ell(\mu))}  \, \Ch_\mu.
\]
(This morphism is denoted $f \mapsto f^\#$ in \cite{LassalleConjecturePQ}.)
Applying this morphism to \eqref{eq:Jp} gives
\begin{align*}
J_\la^\# &= \sum_\mu \a^{-(|\mu|-\ell(\mu))} \, \theta^{(\a)}_\mu(\la) \, \Ch_\mu.
\end{align*}
On the other hand, using self-duality of power-sums and Jack polynomials for the deformed Hall scalar product,
\eqref{eq:Jp} implies
\[ p_\mu = \sum_\la \frac{\langle p_\mu,J_\la \rangle}{\langle J_\la,J_\la \rangle} \, J_\la 
= \sum_\la \frac{\theta^{(\a)}_\mu(\la) \langle p_\mu,p_\mu \rangle}{\langle J_\la,J_\la \rangle} \, J_\la 
= \a^{\ell(\mu)} z_\mu \left( \sum_\la \frac{\theta^{(\a)}_\mu(\la)}{\langle J_\la,J_\la \rangle} \, J_\la \right).\]
Applying $\sh$, we get
\[
    \Ch_\mu  
= \a^{|\mu|} z_\mu \left( \sum_\la \frac{\theta^{(\a)}_\mu(\la)} {\langle J_\la,J_\la \rangle} J^\#_\la
 \right).\]
If we now let $d^{\lambda}_{\mu\nu}$ be defined via the relation
$
\Ch_\mu \Ch_\nu = \sum_\lambda d^{\lambda}_{\mu\nu} \Ch_\lambda,
$
we get
\[ c^{\lambda}_{\mu\nu} \frac{H_\mu \cdot H_\nu}{H_\lambda} = 
    \sum_{\mu',\nu',\la'} \a^{-(|\mu'|-\ell(\mu'))} \, \theta^{(\a)}_{\mu'}(\mu) \, \a^{-(|\nu'|-\ell(\nu'))} \, \theta^{(\a)}_{\nu'}(\nu)
\cdot d^{\lambda'}_{\mu'\nu'} 
\a^{|\la'|} z_{\la'} \frac{\theta^{(\a)}_{\la'}(\la)} {\langle J_\la,J_\la \rangle},
\]
where the sum runs over triplet of partitions $(\mu',\nu',\la')$ of the same size as $(\mu,\nu,\la)$.
But from \cite[Theorem 1.4]{DolegaFerayGaussian}, we get that coefficients $d^{\lambda'}_{\mu'\nu'}$ are polynomials in $\a$
with rational coefficients
(beware that the normalization there is different).
Furthermore, the quantities $\theta^{(\a)}_{\mu'}(\mu)$ are 
also polynomials in $\a$ by Knop--Sahi theorem \cite{KnopSahiCombinatoricsJack}.
Together with the equality $\langle J_\la,J_\la \rangle=H_\la \, H'_{\la}$, we conclude that
\[ \a^{|\mu|+|\nu|-|\la|-2} \, H_\mu \, H_\nu \, H'_\la c^{\lambda}_{\mu\nu} =
\sum_{\mu',\nu',\la'} \a^{\ell(\mu')-1} \, \theta^{(\a)}_{\mu'}(\mu) \, \a^{\ell(\nu')-1} \, \theta^{(\a)}_{\nu'}(\nu)
\cdot d^{\lambda'}_{\mu'\nu'}                                                
 z_{\la'} \theta^{(\a)}_{\la'}(\la)\]
is a polynomial in $\a$ with rational coefficients. 
\end{proof}

\section{A recursion formula for \texorpdfstring{$c^{\lambda}_{\mu\nu}$}{c}}

In this section, we give initial conditions and a recursion formula, determining completely 
the coefficients $c^{\lambda}_{\mu\nu}$.
We do not claim any novelty here. The recursion is already given in \cite{SahiLRCoefShiftedJack},
and the presentation here 
is an adaptation of the technique in \cite{MolevSaganLRFactorialSchur}, 
where the case $\alpha=1$ was established.
\medskip 

First observe that, from the vanishing result by F.~Knop and S.~Sahi and a simple contradiction argument
(see \cite[top of page 4434]{MolevSaganLRFactorialSchur} for the case $\alpha=1$), it follows that
$c^{\lambda}_{\mu\nu}$ is identically zero unless $\lambda \supseteq \mu$ and $\lambda \supseteq \nu$. 
\smallskip

We continue with an easy proposition, that give simple formulas for the $c^{\lambda}_{\mu \lambda}$:
this will be our initial conditions.
\begin{proposition}\label{prop:shiftedLRspecialcase}
We have that $c^{\lambda}_{\mu \lambda} = \jackP^{\#}_\mu(\lambda)$.
\end{proposition}
\begin{proof}
  This follows from evaluating the definition of $c^{\lambda}_{\mu \lambda}$ (\cref{eq:defLR}) in $\lambda$:
\[
\jackP_\mu^\#(\lambda) \jackP_\lambda^\#(\lambda) = \sum_\rho c^{\rho}_{\mu\lambda} \jackP_\rho^\#(\lambda).
\]
Now using the vanishing condition in \eqref{eq:vanishing}, we see that only the term $\rho = \lambda$
survives and from this we can easily deduce the statement.
\end{proof}

The next statement gives the recurrence relation.
When $T$ is a standard Young tableau, we set
$\psi'_T(\alpha) \coloneqq \psi_{T'}(1/\alpha)$, where $T'$ is the transpose of $T$.
\begin{proposition}[The recursion formula] \label{prop:structureConstantRecursion}
Let $\mu, \nu \subseteq \lambda$. Then
\begin{equation}
c^{\lambda}_{\mu\nu} = 
\frac{1}{|\lambda|-|\nu|}\left( 
 \sum_{\nu \to \nu^+} \psi'_{\nu^+ / \nu}(\a) c^{\lambda}_{\mu \nu^+} -
 \sum_{\lambda^- \to \lambda } \psi'_{\lambda / \lambda^-}(\a) c^{\lambda^-}_{\mu \nu}
\right)
\end{equation}
where the first sum is taken over all possible ways to add one box to the diagram $\nu$,
and the second sum is over all ways to remove one box from $\lambda$.
\end{proposition}
The reminder of the section is devoted to the proof of this proposition.
We start with some lemmas.
\begin{lemma}[See \cite{OkounkovOlshanskiShiftedJack}]
If $\mu \subseteq \nu$ then
\[
\frac{  P^{\#}_\mu(\nu) }{ P^{\#}_\nu(\nu) } = 
\frac{1}{(|\nu|-|\mu|)! } \sum_{T \in \SYT(\nu/\mu)} \psi'_T(\alpha),
\]
where the sum is taken over all \emph{standard} Young tableaux of shape $\nu/\mu$.
\end{lemma}

\medskip 

It will be convenient to define $H(\mu,\nu)$ as
\[
H(\mu,\nu) = \frac{  P^{\#}_\mu(\nu) }{ P^{\#}_\nu(\nu) } =
\frac{1}{(|\nu|-|\mu|)! } \sum_{T \in \SYT(\nu/\mu)} \psi'_T(\alpha)
\]
and $H'(\mu,\nu)= (-1)^{|\nu|-|\mu|} H(\mu,\nu)$.
It will be useful to keep in mind that $H(\mu,\mu)=H'(\mu,\mu)=1$.
\begin{lemma}
We have the formula
\begin{equation}\label{eq:StructureCsttsSumTripleProduct}
c_{\mu\nu}^\lambda = \sum_{\nu \subseteq \rho \subseteq \lambda } P^{\#}_\mu(\rho) H(\nu, \rho) H'(\rho, \lambda).
\end{equation}
\label{LemStructureCsttsSumTripleProduct}
\end{lemma}
\begin{proof}
We use induction over $|\lambda|-|\nu|$. The base case $\lambda=\nu$ is given by \cref{prop:shiftedLRspecialcase}.
Substituting $\xvec = \lambda$ in the definition of $c_{\mu\nu}^\lambda$ (\cref{eq:defLR}) and dividing both sides 
with $P^{\#}_\lambda(\lambda)$ together with previous lemma, we obtain

\begin{equation}
 P^{\#}_\mu(\lambda) \frac{P^{\#}_\nu(\lambda) }{P^{\#}_\lambda(\lambda)}  = \sum_{\sigma \subseteq \lambda } c_{\mu\nu}^\sigma  H(\sigma, \lambda).
\end{equation}
Rewriting this now gives 
\begin{equation}\label{eq:rec1}
c_{\mu\nu}^\lambda = P^{\#}_\mu(\lambda)H(\nu, \lambda) - \sum_{\sigma \subsetneq\lambda } c_{\mu\nu}^\sigma  H(\sigma, \lambda).
\end{equation}
The induction hypothesis now allows us to express \eqref{eq:rec1} as 
\begin{align}
c_{\mu\nu}^\lambda &= P^{\#}_\mu(\lambda)H(\nu, \lambda) - 
\sum_{\sigma,\rho:  \mu \subseteq \rho \subseteq \sigma \subsetneq\lambda } 
 P^{\#}_\mu(\rho) H(\nu, \rho) H'(\rho, \sigma)
 H(\sigma, \lambda) \nonumber \\
&= P^{\#}_\mu(\lambda)H(\nu, \lambda) - 
\sum_{\rho: \mu \subseteq \rho \subsetneq\lambda }
 P^{\#}_\mu(\rho) H(\nu, \rho)
\left[ \sum_{\sigma: \rho \subseteq \sigma \subsetneq\lambda } 
 H'(\rho, \sigma) H(\sigma, \lambda)\right] .
 \label{eq:Tech}
\end{align}
It now suffices to prove the following claim: For any $\rho \subseteq \lambda$,
\begin{equation}\label{eq:hproduct}
\sum_{\sigma: \rho \subseteq \sigma \subseteq \lambda } 
 H'(\rho, \sigma) H(\sigma, \lambda) = 0.
\end{equation}
Indeed, if \eqref{eq:hproduct} holds, we can replace the sum in brackets in 
\eqref{eq:Tech} by $- H'(\rho,\la)$ and
we get \eqref{eq:StructureCsttsSumTripleProduct} 
(the first term in \eqref{eq:Tech} correspond to the summand $\rho=\la$ 
in \eqref{eq:StructureCsttsSumTripleProduct}).
\smallskip

Let us now prove \eqref{eq:hproduct}.
Using the definition of $H$, we have that the left-hand side of \eqref{eq:hproduct} is given by
\begin{equation*}
\sum_{\sigma: \rho \subseteq \sigma \subseteq \lambda } \frac{ (-1)^{|\sigma|-|\rho|} }{(|\sigma|-|\rho|)! (|\lambda| - |\sigma|)!}
 \left( 
 \sum_{T_1 \in \SYT(\sigma/\rho)} \psi'_{T_1}(\alpha)
\right)
 \left( 
 \sum_{T_2 \in \SYT(\lambda/\sigma)} \psi'_{T_2}(\alpha)
\right)
\end{equation*}
This triple sum over $\sigma$, $T_1$ and $T_2$ can be rewritten as follows.
First, we sum over standard Young tableaux $T$ of shape $\lambda/\rho$.
This tableau $T$ correspond to a sequence $\rho = \sigma^0 \to \sigma^1 \to \dots \to \sigma^l = \lambda$
which is the ``concatenation'' of $T_1$ and $T_2$.
Then we should sum over $\sigma$ that appear in this sequence,
that is $\sigma=\sigma^k$ for some $k$ between $0$ and $|\lambda|-|\rho|$.
Thus, we get:
\begin{multline*}
\sum_{\rho \subseteq \sigma \subseteq \lambda } 
 H'(\rho, \sigma) H(\sigma, \lambda) = 
\sum_{T \in \SYT(\lambda/\rho) } \psi'_T(\alpha) 
\left[ \sum_{\sigma=\sigma^k \in T}
\frac{ (-1)^{|\sigma|-|\rho|} }{(|\sigma|-|\rho|)! (|\lambda|-|\sigma|)!} \right]\\
=
\left(\sum_{T \in \SYT(\lambda/\rho) } \psi'_T(\alpha) \right) \,
\left(\sum_{k=0}^{|\lambda|-|\rho|}
\frac{ (-1)^{k} }{k! (|\lambda|-|\rho|-k)!}a\right)
\end{multline*}
But, from the binomial theorem, we know that
\[
\sum_{k=0}^{|\lambda|-|\rho|}
\frac{ (-1)^{k} }{k! (|\lambda|-|\rho|-k)!} =0.
\]
Thus \eqref{eq:hproduct} is proved and this completes the proof of 
\cref{LemStructureCsttsSumTripleProduct} by induction.
\end{proof}

We are now ready to prove \cref{prop:structureConstantRecursion}.
\begin{proof}[Proof of \cref{prop:structureConstantRecursion}]
 We fix $\rho$ with $\nu \subseteq \rho \subseteq \lambda$.
 First observe that, using the definition of $H$ and $H'$ we have
\begin{align*}
\sum_{\nu \to \nu^+} \psi'_{\nu^+/\nu} H(\nu^+, \rho) &= (|\rho|-|\nu|) H(\nu, \rho) \text{ and } \\
\sum_{\lambda^- \to \lambda } \psi'_{\lambda/\lambda^-} H'(\rho, \lambda^-) &=
-(|\lambda|-|\rho|) H'(\rho, \lambda).
\end{align*}
Multiplying the first by $H'(\rho,\la)$, the second by $H(\nu,\rho)$
and taking the difference gives
\begin{multline}\label{eq:claim}
  (|\lambda|-|\nu|) H(\nu, \rho) H'(\rho, \lambda) = \\
 \sum_{\nu \to \nu^+} \psi'_{\nu^+/\nu}H(\nu^+, \rho)H'(\rho, \lambda) -
 \sum_{\lambda^- \to \lambda }  \psi'_{\lambda/\lambda^-} H(\nu, \rho)H'(\rho, \lambda^-).
 \end{multline}
Now, simply multiply both sides of the above equation \eqref{eq:claim} with $P^{\#}_\mu(\rho)$, 
and sum them
for all $\rho$ where $\mu \subseteq \rho \subseteq \lambda$.
Using Lemma~\ref{LemStructureCsttsSumTripleProduct} then gives the recursion in  \cref{prop:structureConstantRecursion}.
\end{proof}

\section{Some particular cases}
\label{sec:part}
A weaker form of the conjecture asserts that for any $\alpha>0$ (and any triple
of partitions $\lambda,\mu,\nu$), the coefficient $g^{\lambda}_{\mu,\nu}$,
or equivalently $c^{\lambda}_{\mu,\nu}$, is nonnegative.
This is obviously a strictly weaker statement since a polynomial may take nonnegative
values on nonnegative real numbers but have some negative coefficients,
e.g. $(\alpha-1)^2$.
In this section, we prove that this weaker property holds
whenever $|\lambda|-|\mu| \le 1$.

\begin{proposition}\label{prop:nonnegcase0}
For fixed $\alpha \geq 0$, we have that $g^{\lambda}_{\mu\lambda} \geq 0$.
\end{proposition}
\begin{proof}
    Using \cref{prop:shiftedLRspecialcase} and \cref{EqOkounkovCombJsh}, it is enough to show that
\begin{equation}\label{eq:shiftedproduct}
\sum_{T \in \RSSYT(\mu,n)} \psi_T 
 \left(\prod_{s \in \mu} ( \lambda_{T(s)} - \arm'(s) + \leg'(s)/\alpha) \right)
\end{equation}
is nonnegative. Since $\psi_T$ is nonnegative for all $\alpha>0$, it is enough to show that for each fixed tableau $T$
in the sum above, the product is either positive or $0$.

Fix $T$ and assume that for the box $s \in \mu$, $(\lambda_{T(s)} - \arm'(s) + \leg'(s)/\alpha)<0$.
If $s=(i,j)$, this means
\[
\lambda_{T(i,j)} - (j-1) + (i-1)/\alpha<0.
\]
Since $T$ is a reverse tableau, and $\lambda$ is a partition, $\lambda_{T(i,j)} \geq \lambda_{T(1,j)}$.
This implies that $\lambda_{T(1,j)} - (j-1) <0$. 

Now, we have that $ \lambda_{T(1,l)} \geq \lambda_{T(1,l-1)}$ for any $l$ such that $2\leq l \leq \mu_1$.
As a consequence, the sequence
\[
\lambda_{T(1,j)} - j + 1,\ \lambda_{T(1,j-1)} - j + 2,\ \dots,\ \lambda_{T(1,1)}.
\]
can increase by at most one in each step. Since $\lambda_{T(1,1)} \geq 0$,
we know that at least one of these numbers is $0$.

Thus, every product in \eqref{eq:shiftedproduct} which has a negative factor, also include a factor which is zero.
Hence, $g^{\lambda}_{\mu\lambda} \geq 0$ for fixed $\alpha \geq 0$.
\end{proof}

\begin{proposition}\label{prop:nonnegcase1}
For fixed $\alpha \geq 0$ and $|\lambda|-|\nu|=1$ we have that $c^{\lambda}_{\mu\nu} \geq 0$.
\end{proposition}
\begin{proof}
Using the recursion established above, we have that
\[
c^{\lambda}_{\mu\nu} = \psi'_{\lambda/\nu} \left( c^{\lambda}_{\mu\lambda} - c^{\nu}_{\mu\nu}  \right)
\]
if $|\lambda|-|\nu|=1$. Since $ \psi'_{\lambda/\nu} \geq 0$ when $\alpha>0$, we can use \cref{prop:shiftedLRspecialcase}
and \cref{EqOkounkovCombJsh}
and conclude that it suffices to show that the following is nonnegative, where the sum is over all reverse
semi-standard Young tableaux of shape $\mu$:
\begin{equation}\label{eq:nonnegcase1eq}
\sum_{T} \psi_T \left( 
\prod_{s \in \mu} (\lambda_{T(s)} -  \arm'(s) + \leg'(s)/\alpha)
-
\prod_{s \in \mu} (\nu_{T(s)} -  \arm'(s) + \leg'(s)/\alpha)
\right).
\end{equation}
The reasoning in \cref{prop:nonnegcase0}, implies that if any of the factors in the second product is negative,
then both products contains at least one factor which is zero.
Thus, we can assume that all factors in both products are nonnegative.

Since $\lambda$ and $\nu$ only differ in one coordinate, say $i$, we can factor the difference of the product,
and use that $\lambda_i = \nu_i+1$:
\[
\left( \prod_{\substack{s \in \mu \\ T(s) \neq i}} \nu_{T(s)} -  b(s) \right)
\left(
\prod_{\substack{s \in \mu \\ T(s) = i}} ( \nu_{i} + 1 -  b(s) )
-
\prod_{\substack{s \in \mu \\ T(s) = i}} (\nu_{i}  -  b(s))
\right)
\]
where $b(s) = \arm'(s) + \leg'(s)/\alpha$. 
We assumed that all factors are nonnegative, so the first product is nonnegative.
This fact also implies that the difference in the second parenthesis is nonnegative.
\end{proof}

\begin{remark}
The above technique do not generalize directly to $|\lambda|-|\nu| \geq 2$, 
since some tableaux give negative contributions in that case.
\end{remark}
\bigskip

Let us come back to the original conjecture, stating that
$g^{\lambda}_{\mu,\nu}$ is a Laurent polynomial in $\a$ 
with nonnegative coefficients.
Even in the case $\nu=\la$, where we have
$H_\mu H_\la H'_\la P_\mu^\#(\la)$, we are not aware
of a formula for shifted Jack polynomials explaining the (Laurent) polynomiality
of this quantity.
In \cite{AF18}, we make the following stronger conjecture
(supported by numerical data):
\begin{conjecture} The expression
    $\a^{\ell(\mu)-1} H_\mu P_\mu^\#(x_1,\dots,x_n)$ 
    has nonnegative coefficient in the basis
    \[\bigg( \a^c (x_1-x_2)_{b_1} \cdots (x_{n-1}-x_n)_{b_{n-1}}
    (x_n)_{b_n} \bigg)_{c,b_1,\dots,b_n \ge 0},\]
where $(x)_b$ is as usual the $b^{\text{th}}$ falling power of $x$,
that is $x(x-1) \dotsm (x-b+1)$.
\end{conjecture}
This suggests that an analogue of Knop--Sahi combinatorial formula
\cite{KnopSahiCombinatoricsJack} could exist for shifted Jack polynomials.

\end{document}